\documentclass[11pt]{article}

\usepackage{amssymb, amsmath, amsthm}
\usepackage{enumitem}
\usepackage{mathtools}
\usepackage[onehalfspacing]{setspace}
\usepackage{etoolbox}
\usepackage[a4paper, margin=3cm]{geometry}
\usepackage{color}
\usepackage{cite}
\usepackage{authblk}
\usepackage{nccmath}
\theoremstyle{theorem}
\newtheorem{theorem}{Theorem}[section]

\newtheorem{proposition}[theorem]{Proposition}
\newtheorem{corollary}[theorem]{Corollary}

\theoremstyle{definition}

\newtheorem{example}[theorem]{Example}

\renewenvironment{proof}[1][Proof.]{\begin{trivlist}
		\item[\hskip \labelsep {\bfseries #1}]}{\qed \end{trivlist}}

\setlist{noitemsep, topsep=1ex, parsep=1ex, partopsep=1ex}

\allowdisplaybreaks
\appto\normalsize{
	\abovedisplayskip=2ex plus 1ex minus 1ex
	\belowdisplayskip=2ex plus 1ex minus 1ex
	\abovedisplayshortskip=2ex plus 1ex minus 1ex
	\belowdisplayshortskip=2ex plus 1ex minus 1ex}
\appto\small{
	\abovedisplayskip=2ex plus 1ex minus 1ex
	\belowdisplayskip=2ex plus 1ex minus 1ex
	\abovedisplayshortskip=2ex plus 1ex minus 1ex
	\belowdisplayshortskip=2ex plus 1ex minus 1ex}

\newcommand{\gap}{\vspace{1ex}}

\newcommand{\defn}[1]{\textit{#1}}

\newcommand{\R}{\mathcal{R}}
\newcommand{\Rn}{\mathcal{R}^n}

\newcommand{\Sn}{\mathcal{S}^n}

\newcommand{\V}{\mathcal{V}}

\newcommand{\tr}{\operatorname{tr}}

\newcommand{\Aut}{\operatorname{Aut}}
\newcommand{\Der}{\operatorname{Der}}

\newcommand{\abs}[1]{\left\vert #1 \right\vert}

\newcommand{\Norm}[1]{{\left\vert\kern-0.25ex\left\vert\kern-0.25ex\left\vert #1 
		\right\vert\kern-0.25ex\right\vert\kern-0.25ex\right\vert}}
\newcommand{\ip}[2]{\left< #1,\, #2 \right>}
\newcommand{\set}[2]{\left\lbrace #1 \, : \, #2 \right\rbrace}

\DeclareMathOperator{\dom}{dom}

\newcommand{\llangle}{\langle\mkern-4.5mu\langle}
\newcommand{\rrangle}{\rangle\mkern-4.5mu\rangle}

\title{\bfseries Geometric commutation principle for weakly spectral sets in Euclidean Jordan algebras}

\author[1]{Juyoung Jeong}
\affil[1]{\small
	Department of Mathematics\\
	Changwon National University\\
	Changwon 51140\\
	Republic of Korea\\
	jjycjn@changwon.ac.kr}
\date{\today}

\begin{document}

\maketitle

\begin{abstract}
	A geometric commutation principle in Euclidean Jordan algebra, recently proved by Gowda, says that, for any spectral set $E$ in a Euclidean Jordan algebra $\V$ and $a \in E$, $a$ strongly operator commutes with every element in the normal cone $N_E(a)$. Further, it can be used to establish strong operator commutativity relations in certain optimization problems. Knowing that every spectral sets are special cases of broader class of weakly spectral sets, we prove an analog of a geometric commutation principle for weakly spectral sets and study its consequences and applications.
	
	\vspace{2ex}
	
	\noindent{\bf Key Words:} Euclidean Jordan algebra, weakly spectral sets/functions, commutation principle \\
	\noindent{\bf AMS Subject Classification:} 17C20, 17C30, 52A41, 90C26
\end{abstract}

\vspace{6ex}

%----------------------------------------------------------------------------------------
\section{Introduction}

In optimization theory, we often encounter the problem of optimizing a function of the following form
\[ x \in E \mapsto F(x) + \Phi(x), \]
where $E$ is a constraint set in some underlying space $\V$. Typically $F : \V \to \R$ is the actual objective function we wish to optimize, and $\Phi : \V \to \R$ plays the role of regularizing or giving a penalty to the objective function.

\gap 

In \cite{iusen-seeger-2007}, Iusem and Seeger proved the commutation principle in the setting of $\Sn$, the space of all $n \times n$ real symmetric matrices with $\ip{X}{Y} = \tr(XY)$ for $X, Y \in \Sn$. Let $E$ be a spectral set in $\Sn$ and $\Phi : \Sn \to \R$ be a spectral function. Take $A, B \in \Sn$. If $A$ is a local minimizer/maximizer of the map
\[ X \in E \mapsto \ip{B}{X} + \Phi(X), \]
then $A$ and $B$ commute, i.e., $AB = BA$. Thus, the commutation principle provides a necessary condition for an optimizer of the problem.

\gap 

Realizing $\Sn$ is a primary instance of Euclidean Jordan algebars, Ram{\'i}rez, Seeger, and Sossa \cite{ramirez-seeger-sossa-2013} generalized the commutation principle from $\Sn$ to Euclidean Jordan algebras. Gowda and Jeong \cite{gowda-jeong-2017} further weakened the hypotheses of spectrality in the theorem to (algebra) automorphism invariance, which we call weak spectrality. The commutation principle has since been extended in various directions. For instance, commutation principles in normal decomposition systems by Gowda and Jeong \cite{gowda-jeong-2017}; extended commutation principles in normal decomposition systems by Neizgoda \cite{niezgoda-2018}; and commutation principles in Fan-Theobard-von Neumann systems by Gowda \cite{gowda-2019} to name a few.

\gap 

Recently, Gowda introduced the so-called geometric commutation principle for spectral sets and a commutation principle where $F$ is no longer Fr{\'e}chet differentiable but has $E$-subgradient \cite{gowda-2020}. As a continuation of Gowda's recent work, this paper studies a geometric commutation principle for weakly spectral sets and derives the following results.

\begin{theorem} 
	Let $\V$ be a Euclidean Jordan algebra. If $E$ is a weakly spectral set in $\V$, then $a$ operator commutes with every element in the normal cone $N_{E}(a)$; in particular, with every element in $N_{\langle a \rangle}(a)$.
\end{theorem}

\begin{theorem}
	Let $\V$ be a Euclidean Jordan algebra. Suppose $E$ is a weakly spectral set in $\V$ and $\Phi : \V \to \R$ is a weakly spectral function. If $a \in E$ is a local maximizer of the map
	\[ X \in E \mapsto F(x) + \Phi(x) \]
	and $F : \V \to \R$ has an $E$-subgradient at $a$, then $a$ operator commutes with every element in $\partial_{E} F(a)$. For the minimization problem, we have the same conclusion with $\partial_{E}(-F)(a)$.
\end{theorem}

The organization of the paper is as follows: Section 2 contains definitions and basic properties of Euclidean Jordan algebras. We study geometric commutation principle and its consequences in Section 3 and give examples and applications of the geometric commutation principle in Section 4.

%----------------------------------------------------------------------------------------
\section{Preliminaries}

Throughout the paper, $\V$ represents a Euclidean Jordan algebra of rank $n$, together with the Jordan product $x \circ y$ and inner product $\ip{x}{y} = \tr(x \circ y)$, where $\tr$ is the trace operator. Let $e$ be the unit element in $\V$. The set of the squares $\V_+ = \set{x \circ x}{x \in \V}$ is called the \defn{symmetric cone} of $\V$.

\gap 

Recall that $\V$ is \defn{simple} if it is not a direct product of nonzero Euclidean Jordan algebras.  The classification theorem \cite[Chapter V]{faraut-koranyi-1994} for Euclidean Jordan algebras states that any nonzero Euclidean Jordan algebra is, in a unique way, a direct product of simple Euclidean Jordan algebras and there are only five types of simple Euclidean Jordan algebras up to isomorphism. One of the most prominent simple algebra is $\Sn$, the algebra of $n \times n$ real symmetric matrices with Jordan and inner products given by
\[ X \circ Y = \frac{1}{2}(XY + YX), \;\; \ip{X}{Y} = \tr(XY) \quad \text{for all} \;\; X, Y \in \Sn. \]
In this algebra $\Sn_+$ is the set of all $n \times n$ positive semidefinite matrices. The other four simple algebras are: the algebra of $n \times n$ complex/quaternion Hermitian matrices, the algebra of $3 \times 3$ octonion Hermitian matrices, and the Jordan spin algebra. We refer \cite{faraut-koranyi-1994} for more detail.

\gap 

The $n$ dimensional Euclidean space $\Rn$ can be regarded as a Euclidean Jordan algebra by defining the Jordan product to be the component-wise product, with $\Rn_+$ being the nonnegative orthant of $\Rn$. Note that $\Rn$ is not simple, as it is simply the direct product of $n$ copies of $\mathcal{S}^1$. We say that $\V$ is \defn{essentially simple} if it is either simple or (isomorphic to) $\Rn$.

\gap 

An element $c \in \V$ is an \defn{idempotent} if $c \circ c = c$. If an idempotent $c \neq 0$ cannot be written as the sum of other nonzero idempotents, then it is called \defn{primitive}. The following is the (complete) spectral theorem in Euclidean Jordan algebra \cite[Theorem III.1.2]{faraut-koranyi-1994}. 

\begin{proposition}
	For every $x \in \V$, there exist an (ordered) set of $n$ primitive idempotents $(e_1, e_2, \ldots, e_n)$, called a \defn{Jordan frame}, satisfying
	\[ e_1 + e_2 + \cdots + e_n = e \quad \text{and} \quad e_i \circ e_j = 0 \quad \text{for all} \;\; i \neq j, \]
	and uniquely determined real numbers $\lambda_1(x) \geq \lambda_2(x) \geq \cdots \geq \lambda_n(x)$, called the \defn{eigenvalues} of $x$, such that
	\[ x = \lambda_1(x) e_1 + \lambda_2(x) e_2 + \cdots + \lambda_n(x) e_n. \]
\end{proposition}

Now, we define the \defn{eigenvalue map} $\lambda : \V \to \Rn$ which takes $x \in \V$ to $\lambda(x) \in \Rn$ whose components are the eigenvalues of $x$ arranged in decreasing order. It is well-known that the symmetric cone is the set of all elements whose eigenvalues are nonnegative. Thus, $\V_+ = \set{x \in \V}{\lambda(x) \geq 0}$.

\gap 

A linear transformation $D : \V \to \V$ is a \defn{derivation} if it satisfies
\[ D(x \circ y) = D(x) \circ y + x \circ D(y) \quad \text{for all} \;\; x, y \in \V. \]
The set of all derivations of $\V$ is denoted by $\Der(\V)$. A linear transformation $A : \V \to \V$ is an \defn{algebra automorphism} if it is invertible and preserves the Jordan product, i.e.,
\[ A(x \circ y) = A(x) \circ A(y) \quad \text{for all} \;\; x, y \in \V. \]
The set of all algebra automorphisms of $\V$ is denoted by $\Aut(\V)$. Here, we list some facts which will be used frequently throughout the manuscript.
\begin{itemize}
	\item If $D$ is a derivation on $\V$, then the map $e^{tD}$ is an algebra automorphism of $\V$ for all $t \in \R$, see \cite[page 36]{faraut-koranyi-1994}.
	\item It is known \cite{jeong-gowda-2017} that $\V$ is essentially simple if and only if any Jordan frame can be mapped onto another Jordan frame by an algebra automorphism of $\V$. Hence, for essentially simple algebra $\V$ and $x, y \in \V$ such that $\lambda(x) = \lambda(y)$, there exists $A \in \Aut(\V)$ such that $y = Ax$.
\end{itemize}

We now describe two notions of commutativity concepts in Euclidean Jordan algebras. Given $c \in \V$, define a linear transformation $L_c : \V \to \V$ by $L_c(x) = c \circ x$. Then, for $a, b \in \V$, we say that $a$ and $b$ \defn{operator commute} if $L_aL_b = L_bL_a$, i.e.,
\[ a \circ (b \circ x) = b \circ (a \circ x) \quad \text{for all} \;\; x, \in \V. \]
It is known \cite[Lemma X.2.2]{faraut-koranyi-1994} that $a$ and $b$ operator commute if and only if they are simultaneously diagonalizable, that is, there exists a common Jordan frame $(e_1, e_2, \ldots, e_n)$ such that
\[ a = \sum_{i=1}^{n} a_i e_i, \quad b = \sum_{i=1}^{n} b_i e_i, \]
where $\{ a_1, \ldots, a_n \} = \{ \lambda_1(a), \ldots, \lambda_n(a)\}$ and $\{ b_1, \ldots, b_n \} = \{ \lambda_1(b), \ldots, \lambda_n(b)\}$. Specializing, we say that $a$ and $b$ \defn{strongly operator commute} if there exists a common Jordan frame $(e_1, e_2, \ldots, e_n)$ such that
\[ a = \sum_{i=1}^{n} \lambda_i(a) e_i, \quad b = \sum_{i=1}^{n} \lambda_i(b) e_i. \]
In this case, it is also said that $a, b$ are simultaneously ordered diagonalizable. It is clear that strong operator commutativity implies operator commutativity, but not conversely. For instance, in $\mathcal{S}^2$, two matrices $\begin{bmatrix} 3 & 1 \\ 1 & 3 \end{bmatrix}$ and $\begin{bmatrix*}[r] 2 & -1 \\ -1 & 2 \end{bmatrix*}$ operator commute, but not strongly as they are not simultaneously ordered diagonalizable.

\gap 

For $a \in \V$, we define an \defn{eigenvalue orbit} of $a$ by
\[ [a] = \set{x \in \V}{\lambda(x) = \lambda(a)}. \]
%We use the notation $[x]_{\V}$ when more than one algebra is involved. In the case that $\V$ is not simple, one can write $\V$ as $\V = \V^{(1)} \times \V^{(2)}\times \cdots \times \V^{(k)}$, where each $\V^{(i)}$ is simple. Then, the eigenvalues of any $x = (x^{(1)}, x^{(2)}, \ldots, x^{(k)}) \in \V$ comprise of the eigenvalues of $x^{(i)}$ in $\V^{(i)}$ for $i = 1, 2, \ldots, k$. For such an $x$,  we define the \defn{restricted eigenvalue orbit} by
%\[ [x]_r = \Big\{ \big( y^{(1)}, y^{(2)}, \ldots, y^{(k)} \big) \in \V : y^{(i)} \in [x^{(i)}]_{\V^{(i)}} \; \text{for all} \;\; i \Big\}. \]
%We note that $[x]_r\subseteq [x]$ with equality holding when $\V$ is simple. To see an example, we have $[(1, 0)]=\{(1, 0), (0, 1)\}$, while $[(1, 0)]_r=\{(1, 0)\}$ in $\R^2$. 
Similarly, for $a \in \V$, an \defn{automorphism orbit} and \defn{$\varepsilon$-restricted automorphism orbit} of $x$ are respectively defined by
\[ \langle a \rangle = \set{Aa}{A \in \Aut(\V)} \quad \text{and} \quad \llangle a \rrangle_{\varepsilon} = \set{e^{tD}a}{D \in \Der(\V), \abs{t} < \varepsilon}. \]
Since $e^{tD} \in \Aut(\V)$ for all $D \in \Der(\V)$ and $t \in \R$, it is evident that  $\llangle a \rrangle_{\varepsilon} \subseteq \langle a \rangle$. Also, as automorphisms preserve the eigenvalues, we also have $\langle a \rangle \subseteq [a]$ with equality holding if and only if $\V$ is essentially simple. 

\gap 

Lastly, we consider spectrality in Euclidean Jordan algebras. For any (permutation invariant) set $Q$ in $\Rn$, the set $E = \lambda^{-1}(Q) \in \V$ is called a \defn{spectral set}. A set $E$ in $\V$ is said to be a \defn{weakly spectral set} if $A(E) \subseteq E$ for all $A \in \Aut(\V)$; thus $E$ is invariant under all algebra automorphisms of $\V$. From the definition, we see that $x \in E \Rightarrow [x] \subseteq E$ if and only if $E$ is spectral, and that $x \in E \Rightarrow \langle x \rangle \subseteq E$ if and only if $E$ is weakly spectral. Similarly, for any (permutation invariant) function $f : \Rn \to \R$, the function $F = f \circ \lambda$ is called a \defn{spectral function} from $\V$ to $\R$. A function $F : \V \to \R$ is called a \defn{weakly spectral function} if $F(Ax) = F(x)$ for all $x \in \Aut(\V)$. Recall that, when $\V = \Sn$, spectrality and weak spectrality are equivalent. This fact has been generalized to essentially simple algebras \cite{jeong-gowda-2017}, which is described below. 

\begin{proposition}
	Every spectral set/function in $\V$ is weakly spectral.	Additionally, the converse holds if and only if $\V$ is essentially simple.
\end{proposition}

%----------------------------------------------------------------------------------------
\section{Geometric commutation principle for weakly spectral sets}

Given a nonempty set $S \subseteq \V$, the \defn{normal cone} of $S$ at $a$ is defined by
\[ N_{S}(a) := \set{d \in \V}{\ip{d}{x-a} \leq 0 \; \text{for all} \; x \in S}. \]
Let $F : \V \to (-\infty, \infty]$, $S \subseteq \V$, and $a \in S \cap \operatorname{dom} F$. The \defn{subdifferential} of $F$ at $a$ relative to $S$ is defined by
\[ \partial_{S} F(a) := \set{d \in \V}{F(x) - F(a) \geq \ip{d}{x-a} \; \text{for all} \; x \in S}. \]
Any element of $\partial_{S} F(a)$ is called an \defn{$S$-subgradient} of $F$ at $a$.
When $S = \V$, we simply write $\partial F(a)$ in place of $\partial_{\V} F(a)$. It is well known that if $F$ if Fr{\'e}chet differentiable at $a$, then $\partial F(a) = \{\nabla F(a)\}$.

\gap

We now state an analog of the geometric commutation principle for weakly spectral sets. While the item $(a)$ of the theorem below is a restatement of the Theorem 2 in \cite{gowda-2020}, it is included here to emphasize the parallel structure between two items. The key idea of the proof is borrowed from \cite{gowda-jeong-2017}.

\begin{theorem} \label{thm: geometric commutation principle for weakly spectral sets}
	In a Euclidean Jordan algebra $\V$, suppose $E \subseteq \V$ and $a \in E$.
	\begin{itemize}
		\item[$(a)$] If $E$ is spectral, then $a$ strongly operator commutes with every element in $N_{[a]}(a)$; hence, with every element in $N_{E}(a)$.
		\item[$(b)$] If $E$ is weakly spectral, then $a$ operator commutes with every element in $N_{\llangle a \rrangle_{\varepsilon}}(a)$ for all $\varepsilon > 0$; hence, with every element in $N_{\langle a \rangle}(a)$ and $N_{E}(a)$.
	\end{itemize}
	Moreover, when $\V$ is essentially simple, strong operator commutativity is obtained in $(b)$.
\end{theorem}

\begin{proof}
	$(a)$ The proof can be found in \cite{gowda-2020}.
	
	\gap 
	
	$(b)$ Since $E$ is weakly spectral, we have $\llangle a \rrangle_{\varepsilon} \subset \langle a \rangle \subseteq E$; hence $N_{E}(a) \subseteq N_{\langle a \rangle}(a) \subset N_{\llangle a \rrangle_{\varepsilon}}(a)$. Thus, it is sufficient to show that $a$ operator commutes with every elements in $N_{\llangle a \rrangle_{\varepsilon}}(a)$. Choose $d \in N_{\llangle a \rrangle_{\varepsilon}}(a)$ so that $\ip{d}{x - a} \leq 0$ for all $x \in \llangle a \rrangle_{\varepsilon}$. This implies that 
	\[ \ip{e^{tD}a}{d} \leq \ip{a}{d} \quad \text{for all} \;\; D \in \Der(\V), \abs{t} < \varepsilon. \]
	Since the map $t \mapsto \ip{e^{tD}a}{d}$ attains its maximum at $t = 0$, its derivative at $t = 0$ is zero, which implies $\ip{Da}{d} = 0$. Now, for any $u, v \in \V$, the map $D := L_{u}L_{v} - L_{v}L_{u}$ is known to be a derivation \cite[Proposition II.4.1]{faraut-koranyi-1994}. Substituting this map into $\ip{Da}{d} = 0$ yields
	\[ \ip{(L_{u}L_{v} - L_{v}L_{u})a}{d} = 0. \]
	Note that $L_{u}$ and $L_{v}$ are self-adjoint and $L_{x}(y) = L_{y}(x)$ for all $x, y \in \V$. Thus, an easy calculation shows that the above equality is equivalent to
	\[ \ip{(L_{a}L_{d} - L_{d}L_{a})u}{v} = 0. \]
	Since $u, v$ can be chosen arbitrarily, we must have $L_{a}L_{d} - L_{d}L_{a} = 0$, implying that $a$ and $d$ operator commute. As $d \in N_{\llangle a \rrangle_{\varepsilon}}(a)$ is arbitrary, $a$ operator commutes with every elements in $N_{\llangle a \rrangle_{\varepsilon}}(a)$, completing the proof.
\end{proof}

\begin{corollary}
	For any idempotent $c \in \V$, $y \in N_{\V_+}(c)$ if and only if $-y \in \V_+$ and $\ip{y}{c} = 0$.
\end{corollary}

\begin{proof}
	Note that $\V_+$ is (weakly) spectral set in $\V$. First, assume that $-y \in \V_+$ and $\ip{y}{c} = 0$, then choose arbitrary $x \in \V_+$. Since $\lambda(x) \geq 0$ and $\lambda(y) \leq 0$, we see that
	\[ \ip{y}{x-c} = \ip{y}{x} \leq \ip{\lambda(y)}{\lambda(x)} \leq 0, \]
	where the first inequality follows from von Neumann-type trace inequality in Euclidean Jordan algebra \cite{baes-2007}. Hence, $y \in N_{\V_+}(c)$.
	
	\gap 
	
	Conversely, suppose $y \in N_{\V_+}(c)$, so that $\ip{y}{x - c} \leq 0$ for all $x \in \V_+$. Then, $y$ operator commutes with $c$ by the previous theorem. Thus, we can write
	\[ y = \sum_{i=1}^{n} y_i e_i, \quad c = e_1 + \cdots + e_k. \]
	Then, for $j = 1, \ldots, n$, putting $x = e_j + c \in \V_+$ and using the orthonormality of $(e_1, \ldots, e_n)$ give 
	\[ 0 \leq \ip{y}{(e_j + c) - c} = \ip{y}{e_j} = y_j. \]
	This means that all the eigenvalues of $y$ are nonnegative, thus $-y \in \V_+$. Moreover, plugging $x = 0, 2c \in \V_+$ into $\ip{y}{x - c} \leq 0$ gives $\ip{y}{c} = 0$, as desired. 
\end{proof}

\begin{corollary}
	In a Euclidean Jordan algebra $\V$, let $d \in \partial F(a)$ for some function $F : \V \to \R$.
	\begin{itemize}
		\item[$(a)$] If $F$ is spectral, then $d$ and $a$ strongly operator commute.
		\item[$(b)$] If $F$ is weakly spectral, then $d$ and $a$ operator commute.
	\end{itemize}
	Moreover, when $\V$ is essentially simple, strong operator commutativity is obtained in $(b)$.
\end{corollary}

\begin{proof}
	Suppose $d \in \partial F(a)$. By the definition of the subgradient,
	\[ F(x) - F(a) \geq \ip{d}{x-a} \quad \text{for all} \;\; x \in \V. \]
	
	$(a)$ In the case that $F$ is spectral, $\dom F$ is a spectral set, thus we have $[a] \subseteq \dom F$. In particular, the inequality above is true for all $x \in [a]$. Now, since $F(x) = F(a)$ for all $x \in [a]$, we get $0 \geq \ip{d}{x-a}$ for all $x \in [a]$. It follows that $d \in N_{[a]}(a)$; hence $d$ and $a$ strongly operator commute by Theorem \ref{thm: geometric commutation principle for weakly spectral sets} $(a)$.
	
	\gap 
	
	$(b)$ Note that $\dom F$ is a weakly spectral set as $F$ is weakly spectral. Thus, $\langle a \rangle \subseteq \dom F$. Since any $x \in \langle a \rangle$ can be written as $x = Aa$ for some $\Aut(\V)$ and $F(Aa) = F(a)$. It follows that  $0 \geq \ip{d}{x-a}$ for all $x \in \langle a \rangle$, meaning that $d \in N_{\langle a \rangle}(a)$. Hence $d$ and $a$ operator commute by Theorem \ref{thm: geometric commutation principle for weakly spectral sets} $(b)$.
\end{proof}

The following theorem extends Theorem 3 in \cite{gowda-2020}. We include the result as part $(a)$ below to highlight a similarity between two items. Also, a similar argument to the following theorem has been obserbed in \cite{jeong-david-2024}.

\begin{theorem} \label{thm: commutation principle with weak spectrality}
	Let $\V$ be a Euclidean Jordan algebra. Consider the optimization problem
	\[ \max_{x \in E} F(x) + \Phi(x), \]
	where $E \subseteq \V$, $\Phi : \V \to \R$, and $F : \V \to \R$ has an $E$-subgradient at $a$.
	\begin{itemize}
		\item[$(a)$] If $a \in E$ is a global maximizer, $E$ is spectral and has a $E$-subgradient at $a$, and $\Phi$ is a spectral function, then $a$ strongly operator commutes with every elements in $\partial_{[a]} F(a)$, in particular with those in $\partial_E F(a)$ and $\partial F(a)$.
		\item[$(b)$] If $a \in E$ is a local maximizer, $E$ is weakly spectral and has a $E$-subgradient at $a$, and $\Phi$ is a spectral function, then $a$ operator commutes with every elements in $\partial_{\langle a \rangle} F(a)$, in particular with those in $\partial_E F(a)$ and $\partial F(a)$.
	\end{itemize}
	For the minimization problem, the same conclusion holds with $\partial_{E}(-F)(a)$ in both items $(a)$ and $(b)$. Moreover, when $\V$ is essentially simple, we obtain strong operator commutativity in $(b)$.
\end{theorem}

\begin{proof}
	$(a)$ The proof can be found in \cite{gowda-2020}.
	
	\gap 
	
	$(b)$ Since $a$ is a local maximizer of the given problem, we have
	\[ F(a) + \Phi(a) \geq F(x) + \Phi(x) \quad \text{for all} \;\; x \in N_a \cap E, \]
	where $N_a$ denotes an open ball around $a$. Since $a \in E$ and $E$ is weakly spectral, we see that $Aa \in E$ for all $A \in \Aut(\V)$. In particular, with fixed derivation $D$, we have $e^{tD}a \in E$ for all $t \in \R$. Since the map $t \mapsto e^{tD}a$ is continuous and has the value $a$ at $t = 0$, there exists $\varepsilon > 0$ such that $e^{tD}a \in N_a$ for all $\abs{t} < \varepsilon$. It follows that $\llangle a \rrangle_{\varepsilon} \subseteq N_a \cap E$ from the previous observations.
	Note that $\Phi(x) = \Phi(a)$ for all $x \in \llangle a \rrangle_{\varepsilon}$. Then
	\[ F(a) \geq F(x) \quad \text{for all} \;\;\; x \in \llangle a \rrangle_{\varepsilon}. \]
	Now, we show that $a$ operator commutes with $d \in \partial_{\langle a \rangle} F(a)$. Firstly, we see that  
	\[ F(x) - F(a) \geq \ip{d}{x - a} \quad \text{for all} \;\; x \in \langle a \rangle. \]
	Since $\llangle a \rrangle_{\varepsilon} \subseteq \langle a \rangle$ and $\langle a \rangle$ is weakly spectral, the inequality reduces to $\ip{d}{x - a} \leq 0$ for all $x \in \llangle a \rrangle_{\varepsilon}$, meaning $a \in N_{\llangle x \rrangle_{\varepsilon}}(a)$. Hence, by Theorem \ref{thm: geometric commutation principle for weakly spectral sets} $(b)$, $a$ operator commutes with $d$.
\end{proof}

\begin{example}
	In Theorem \ref{thm: commutation principle with weak spectrality} $(a)$, the assumption that $a$ is a global maximizer cannot be relaxed to allow $a$ to be merely a local maximizer. To see this, consider $\V = \R^2$ and
	\[ \max_{(x, y) \in E} F(x, y) = (x-1)^2 + (y-2)^2, \]
	where $E = \set{(x, y) \in \R^2}{x, y \geq 0, x+y = 3}$. Note that $(3, 0) \in E$ is a global maximizer with the value $F(3, 0) = 8$, and $(0, 3) \in E$ is a local maximizer with the value $F(0, 3) = 2$. We now show that a local maximizer $(0, 3)$ does not strongly operator commute with some elements in $\partial_{[\{(0, 3)\}]} F(0, 3)$. Indeed, a straightforward calculation reveals that
	\[ \partial_{[\{(0, 3)\}]} F(0, 3) = \set{(x, y) \in \R^2}{x \leq y + 2}. \]
	Therefore, we have $(2, 0) \in \partial_{[\{(0, 3)\}]}F(0, 3)$, while $(0, 3)$ does not strongly operator commute with $(2, 0)$. On the other hand, for a global maximizer $(3, 0)$, we have
	\[ \partial_{[\{(3, 0)\}]} F(3, 0) = \set{(x, y) \in \R^2}{x \geq y + 2}. \]
	Thus, $(3, 0)$ strongly operator commutes with every element in $\partial_{[\{(3, 0)\}]}F(3, 0)$.
\end{example}

It is noteworthy to mention that, when $\V$ is essentially simple, one can always expect strong operator commutativity in the geometric commutation principle even with weak spectrality assumption. However, this is not always the case for general algebra $\V$.

\begin{example}
	Let $\V = \mathcal{S}^1 \times \mathcal{S}^2$ and take 
	\[ A := \left[ \begin{array}{@{}c|cc@{}}
		2 & 0 & 0 \\ \hline 0 & 2 & 1 \\ 0 & 1 & 2
	\end{array} \right], \quad B := \left[ \begin{array}{@{}c|rr@{}}
		2 & 0 & 0 \\ \hline 0 & 2 & -1 \\ 0 & -1 & 2
	\end{array} \right], \quad C := \left[ \begin{array}{@{}c|cc@{}}
		4 & 0 & 0 \\ \hline 0 & 2 & 1 \\ 0 & 1 & 2
	\end{array} \right]. \]
	Consider the optimization problem $\displaystyle \max_{X \in \langle B \rangle} \ip{C}{X}$. Then, we see that $A \in \langle B \rangle$ is a (global) maximizer to this problem. Note that $\langle B \rangle$ is a weakly spectral set and that $A$ operator commutes with $C$ but not strongly.
\end{example}

%\begin{example}
%	For $f : \V \to (-\infty, \infty]$, the proximal operator of $f$ is given by
%	\[ \operatorname{prox}_f(a) = \argmin_{x \in \V} \left\{ \frac{1}{2} \norm{x-a}^2 + f(x) \right\} \;\; \text{for any} \;\; a \in \V. \]
%	Suppose $f$ is spectral (weakly spectral, resp.) and $p \in \operatorname{prox}_f(a)$. Then $p$ and $a$ strongly operator commute (operator commute, resp.). To see this, note that
%	\[ p = \argmax_{x \in \V} \left\{ -\frac{1}{2} \norm{x-a}^2 - f(x) \right\} \]
%	Define $F(x) = -\frac{1}{2} \norm{x-a}^2$. Since $F$ is differentiable, $p$ strongly operator commute (operator commute, resp.) with $\nabla F(p) = a-p$; hence with $(a-p)+p = a$.
%\end{example}

\section*{Acknowledgment}
This research was supported by Changwon National University in 2023--2024.




\end{document}